\documentclass[a4paper,leqno,11pt]{amsart}
\usepackage{amsfonts,amssymb,verbatim,amsmath,amsthm,latexsym,textcomp,amscd}
\usepackage{latexsym,amsfonts,amssymb,epsfig,verbatim}
\usepackage{lipsum}
\usepackage{amsmath,amsthm,amssymb,latexsym,graphics,textcomp}
\usepackage{graphicx}
\usepackage{color}
\usepackage{url}
\usepackage{enumerate}
\usepackage[mathscr]{euscript}

\usepackage{hyperref}

\theoremstyle{plain}
\newtheorem{theorem}{Theorem}[section]
\newtheorem{thm}[theorem]{Theorem}
\newtheorem{lemma}[theorem]{Lemma}
\newtheorem{cor}[theorem]{Corollary}
\newtheorem{prop}[theorem]{Proposition}

\newtheorem{example}[theorem]{Example}
\newtheorem{remark}[theorem]{Remark}

\theoremstyle{definition}

\newtheorem{defn}[theorem]{Definition}
\newtheorem{question}[theorem]{Question}
\newtheorem{rmk}[theorem]{Remark}
\newtheorem{exam}[theorem]{Example}
\begin{document}
\title{On Leibniz algebras whose centralizers are ideals}
\author{Pratulananda Das}
\email{pratulananda@yahoo.co.in}
\address{Department of Mathematics,
Jadavpur University, Kolkata-700032, West Bengal, India.}

\author{Ripan Saha}
\email{ripanjumaths@gmail.com}
\address{Department of Mathematics,
Raiganj University,
College para, Uttar Dinajpur, PIN-733134, West Bengal, India.}

\begin{abstract}
This paper concerns the study of Leibniz algebras, a natural generalization of Lie algebras, from the perspective of centralizers of elements. We study conditions on Leibniz algebras under which centralizers of all elements are ideals. We call a Leibniz algebra, a CL-algebra if centralizers of all elements are ideals. We discuss nilpotency of CL-algebras.
\end{abstract}
\maketitle
\section{Introduction}
In \cite{LP}, J.-L. Loday introduced some new types of a non-anticommutative version of Lie algebras along with their (co)homologies. Historically, such algebraic structures had been studied by A. Bloh who called them D-algebras \cite{Bloh}. In this new type of algebras, the bracket satisfies Leibniz identity instead of Jacobi identity and such algebras are called Leibniz algebras.

In group theory, centralizers of a group give many useful pieces of information about the structure of the group. Investigations of centralizers is an interesting research area in group theory \cite{A, BS}. Centralizers of Lie algebras are studied in many articles \cite{BI, Z, SMR}. The present article contributes to the development of the theory of non-associative algebras from the perspective of centralizers of Leibniz algebras. Leibniz algebras whose subalgebras are ideals are studied in \cite{KSS}.  In this article, we first study conditions on Leibniz algebras for which centralizers of all elements of Leibniz algebras are ideals. Leibniz algebras which satisfy the required conditions will be called CL-algebras. The main question addressed in this paper is whether a finite-dimensional nilpotent Leibniz algebra is a CL-algebra or not. We show in Theorem \ref{main theorem nil} that any nilpotent complex Leibniz algebra up to dimension four is a CL-algebra. We define a notion of CL-property in a Leibniz algebra. Elements which satisfy the CL-property are called CL-elements. In \cite{MS}, Mukherjee and Saha have introduced a notion of a finite group action on a Leibniz algebra. We show that a CL-element is preserved under the group actions on Leibniz algebras. At the end of this article, we mention some questions for further study.
\section{Preliminaries}
In this section, we discuss some basics of Leibniz algebras.
\begin{defn}   
Let $\mathbb{K}$ be a field. A Leibniz algebra is a vector space $L$ over $\mathbb{K},$ equipped with a bracket operation, which is $\mathbb{K}$-bilinear and satisfies the Leibniz identity: 
$$[x,[y,z]]= [[x,y],z]-[[x,z],y] ~~\mbox{for}~x,~y,~z \in L.$$
\end{defn}
Note that any Lie algebra is automatically a Leibniz algebra as in the presence of skew symmetry the Leibniz identity is the same as the Jacobi identity. A Leibniz algebra $(L, [~,~])$ is called an abelian algebra if $[x,y]=0$ for all $x, y\in L$.

A morphism between two Leibniz algebras $L_1$ and $L_2$ is a $\mathbb{K}$-linear map $f:L_1\to L_2$ which satisfies $f([x,y])=[f(x),f(y)]$ for all $x,y\in L_1$.
\begin{exam}\label{example-1}
Consider a differential Lie algebra $(L,d)$ with the Lie bracket $[~,~]$. Then $L$ has a  Leibniz algebra structure with the bracket operation $[x,y]_d:= [x,dy]$. The new bracket on $L$ is called the derived bracket.
\end{exam}
\begin{defn}
A Leibniz algebra $L$ is said to be simple if it contains only the following ideals: $\lbrace 0 \rbrace, I, L$. Here $I$ denotes ideal generated by elements of the form $[x,x]$ for all $x\in L$.
\end{defn}
Given a Leibniz algebra $L$, we define the following two sided ideals:
\begin{align*}
&L^1=L\,\,\,\text{and}\,\,\,L^{k+1}=[L^k,L],\\ 
&L^{[1]}=L\,\,\,\text{and}\,\,\,L^{[k+1]}=[L^{[k]},L^{[k]}]\,\,\, \text{for}\,\,\,k\geq 1.
\end{align*}
\begin{defn}
A Leibniz algebra $L$ is called a nilpotent Leibniz algebra if there exist $n\in \mathbb{N}$ such that $L^n=0$. Suppose $n\in \mathbb{N}$ is least and $L^n=0$ then $L$ is called an $n$-step nilpotent Leibniz algebra.
\end{defn}
\begin{exam}\label{ex}\label{example-2}
\cite{A3} Consider a three dimensional vector space $L$ spanned by \linebreak$\{e_1,~e_2,~e_3\}$ over $\mathbb{C}$. Define a bilinear map $[~,~]: L\times L \longrightarrow L$ by $[e_1,e_3]=e_2$ and $[e_3,e_3]= e_1$, all other products of basis elements being $0$. Then $(L,[~,~])$ is a Leibniz algebra over $\mathbb{C}$ of dimension $3$. The Leibniz algebra $L$ is  nilpotent and is denoted by $\lambda_6$  in the classification of three dimensional nilpotent Leibniz algebras. 
\end{exam}
\begin{defn}
A Leibniz algebra $L$ is called a solvable Leibniz algebra if there exist $n\in \mathbb{N}$ such that $L^{[n]}=0$. 
\end{defn}
Note that any nilpotent Leibniz algebra is a solvable Leibniz algebra but the converse is not true in general. For example, the subalgebra of 
$\mathfrak{gl}(n, \mathbb{R}),~n\geq 2$, consisting of upper triangular matrices, ${\mathfrak {b}}(n,\mathbb {R} )$, is solvable but not nilpotent \cite[p.~10]{Humphreys}.
\begin{defn}
An ideal $P$ of a nilpotent Leibniz algebra $L$ is called nilpotent if $P$ is nilpotent as a Leibniz algebra.
\end{defn}
\subsection{Classsification of nilpotent Leibniz algebras upto dimension 4}
Here we recall some well-known classification results for nilpotent Leibniz algebras.
\begin{theorem}\cite{Loday}\label{dim2}
Let $L$ be a $2$-dimensional nilpotent Leibniz algebra. Then $L$ is either abelian or isomorphic to
$$\mu_1: [e_1,e_1]=e_2.$$
\end{theorem}
\begin{theorem}\cite{AAO}\label{dim3}
Let $L$ be a $3$-dimensional nilpotent Leibniz algebra. Then $L$ is isomorphic to one of the following pairwise non-isomorphic algebras:
\begin{align*}
&\lambda_1: \text{abelian},\\
&\lambda_2:[e_1,e_1]=e_3,\\
&\lambda_3:[e_1,e_2]=e_3,[e_2,e_1]=-e_3,\\
&\lambda_4(\alpha): [e_1,e_1]=e_3,[e_2,e_2]=\alpha e_3,[e_1,e_2]=e_3,\\
&\lambda_5:[e_2,e_1]=e_3,[e_1,e_2]=e_3,\\
&\lambda_6:[e_1,e_1]=e_2,[e_2,e_1]=e_3.
\end{align*}
\end{theorem}
\begin{theorem}\label{dim4}\cite{AOR} There exist up to isomorphism five one parametric families and twelve concrete representatives of nilpotent complex Leibniz algebras of dimension four, namely:
$$\rho_1: [e_1,e_1]=e_2, [e_2,e_1]=e_3,
[e_3,e_1]=e_4;$$
$$\rho_{2}: [e_1,e_1]=e_3, [e_1,e_2]=e_4,
[e_2,e_1]=e_3, [e_3,e_1]=e_4;$$
$$\rho_{3}: [e_1,e_1]=e_3, [e_2,e_1]=e_3, [e_3,e_1]=e_4;$$
$$\rho_4(\alpha): [e_1,e_1]=e_3, [e_1,e_2]=\alpha e_4,
[e_2,e_1]=e_3, [e_2,e_2]=e_4, [e_3,e_1]=e_4, \quad \alpha\in
\{0,1\};$$
$$\rho_{5}:  [e_1,e_1]=e_3, [e_1,e_2]=e_4, [e_3,e_1]=e_4;$$
$$\rho_6: [e_1,e_1]=e_3, [e_2,e_2]=e_4, [e_3,e_1]=e_4 ;$$
$$\rho_7: [e_1,e_1]=e_4, [e_1,e_2]=e_3, [e_2,e_1]=- e_3, [e_2,e_2]= -2e_3 +e_4 ;$$
$$\rho_8: [e_1,e_2]=e_3, [e_2,e_1]=e_4, [e_2,e_2]=-e_3;$$
$$\rho_9(\alpha):  [e_1,e_1]=e_3, [e_1,e_2]=e_4, [e_2,e_1]=-\alpha e_3,
[e_2,e_2]=-e_4, \quad \alpha\in \mathbb{C};$$
$$\rho_{10}(\alpha):  [e_1,e_1]=e_4, [e_1,e_2]=\alpha e_4,
[e_2,e_1]=-\alpha e_4, [e_2,e_2]=e_4, [e_3,e_3]=e_4, \quad
\alpha\in \mathbb{C};$$
$$\rho_{11}: [e_1,e_2]=e_4, [e_1,e_3]=e_4, [e_2,e_1]=-e_4, [e_2,e_2]=e_4, [e_3,e_1]=e_4;$$
$$\rho_{12}: [e_1,e_1]=e_4, [e_1,e_2]=e_4, [e_2,e_1]=-e_4, [e_3,e_3]=e_4;$$
$$\rho_{13}: [e_1,e_2]=e_3, [e_2,e_1]=e_4;$$
$$\rho_{14}: [e_1,e_2]=e_3, [e_2,e_1]=-e_3, [e_2,e_2]=e_4;$$
$$\rho_{15}: [e_2,e_1]=e_4, [e_2,e_2]=e_3;$$
$$\rho_{16}(\alpha): [e_1,e_2]=e_4, [e_2,e_1]=(1+\alpha)/(1-\alpha)e_4,
[e_2,e_2]=e_3, \quad \alpha\in {\mathbb{C}\backslash\{1\}};$$
$$\rho_{17}: [e_1,e_2]=e_4, [e_2,e_1]=-e_4, [e_3,e_3]= e_4.$$
\end{theorem}
\section{CL-algebras}
In this section, we introduce conditions so that for a Leibniz algebra satisfying these conditions centralizers of each element are ideals. We introduce a subclass of collection of all Leibniz algebras, a member of this subclass is called a CL-algebra. We give some examples of CL-algebras and show that the centralizer of an element in a CL-algebra is an ideal (cf. Theorem \ref{ideal}).
\begin{defn}
Let $L$ be a Leibniz algebra over a field $\mathbb{K}$. We define centralizer of $x\in L$ as follows:
$$C_L(x)=\lbrace y\in L \mid [x,y]=0=[y,x]\rbrace.$$
The right centralizer of $x\in L$ is defined as 
$$C^r_L(x)=\lbrace y\in L \mid [x,y]=0\rbrace.$$
The left centralizer of $x\in L$ is defined as 
$$C^l_L(x)=\lbrace y\in L \mid [y,x]=0\rbrace.$$
So, the centralizer of $x\in L$ is both left and right centralizer of $x$.
\end{defn}
\begin{rmk}
In case of a Lie algebra both the left and the right centralizers are same. Observe that for any $x\in L$, the square element $[x,x]\in C^r_L(x)$ but $[x,x]$ may not belongs to $C^l_L(x)$. For example, consider the algebra with multiplication $[e_3, e_3]=e_1, [e_1,e_3]=e_2$, and all other products are zero. We have $C_L(e_3)=~<e_2>$, and $[e_3, e_3]=e_1 \notin C^l_L(e_3)$.
\end{rmk}
\begin{lemma}
For any Leibniz algebra $L$ and any $x\in L$, $C_L(x)$ is a Leibniz subalgebra.
\end{lemma}
\begin{proof}
Let $y_1,y_2\in C_L(x)$. From the bilinearity of bracket operation, $y_1-y_2\in C_L(x)$. Now,
\begin{align*}
[x,[y_1,y_2]]&=[[x,y_1],y_2]-[[x,y_2],y_1]\\
                    &=[0,y_2]-[0,y_1]\\
                    &=0.
\end{align*}
Thus, $[y_1,y_2]\in C_L(x)$. Similarly, $[y_2,y_1]\in C_L(x)$. So, $C_L(x)$ is a Leibniz subalgebra.
\end{proof}
\begin{defn}\label{CL}
A Leibniz algebra $L$ is called a left CL-algebra if it satisfies the following conditions:
\begin{align}
&[[x,a],y]=0,\\
&[[a,x],y]=0, \,\,\,\forall\, y\in C_L(x),~ \text{and}~a,x\in L.
\end{align}
A Leibniz algebra $L$ is called a right CL-algebra if it satisfies the following conditions:
\begin{align}
&[[x,a],y]=0,\\
&[y,[a,x]]=0, \,\,\,\forall\, y\in C_L(x),~ \text{and}~a,x\in L.
\end{align}
\end{defn}
\begin{defn}
A Leibniz algebra $L$ is called a CL-algebra if it is both a left and a right CL-algebra.
\end{defn}
\begin{remark}\label{CL-3}
To verify a Leibniz algebra is a CL-algebra, we need to check the following conditions for all $x\in L$,
\begin{enumerate}
\item $[[x,L],C_L(x)]=0$,
\item $[[L,x],C_L(x)]=0$,
\item $[C_L(x),[L,x]]=0$.
\end{enumerate}
\end{remark}
\begin{example}
Any Abelian Leibniz algebra is a CL-algebra. The following example shows that the converse part may not be true.
\end{example}
\begin{example}
Let $L$ be a vector space of dimension $2$ over a field $F$ and $\lbrace a,b\rbrace$ be a basis of $L$. Define the bracket $[~,~]$ by the following rule:
$[a,a]=b,\,\,\,[b,a]=[b,b]=[a,b]=0$. It is enough to check the conditions of CL-algebras for the basis elements. For the basis elements, we have,
\begin{align*}
&C_L(a)=~<b>,\,\, [[a,L],C_L(a)]=0,\,\, [[L,a],C_L(a)]=0,\,\, [C_L(a),[L,a]]=0;\\
&C_L(b)=L,\,\, [[b,L],C_L(b)]=0,\,\, [[L,b],C_L(b)]=0,\,\, [C_L(b),[L,b]]=0.
\end{align*}
Thus, $L$ is a CL-algebra, which is non-abelian.
\end{example}
\begin{example}
Any $3$-step nilpotent Leibniz algebra $L$ is a CL-algebra. As $L$ is a $3$-step nilpotent Leibniz algebra, we have $L^3=0$. Thus, by Definition \ref{CL} $L$ is a CL-algebra.
\end{example}
\begin{theorem}\label{ideal}
Suppose $L$ is a Leibniz algebra and $x\in L$. Centralizers $C_L(x)$ in L are left (right) ideal of $L$ if and only if $L$ is a left (right) CL-algebra.
\end{theorem}
\begin{proof}
Suppose $L$ is a left CL-algebra. Let $y\in C_L(x)$ and $a\in L$. We show that $[a,y]\in C_L(x)$.
\begin{align*}
[[x,[a,y]&=[[x,a],y]-[[x,y],a]\\
             &=-[0,a]=0.
\end{align*}
Similary, using the second equation of left CL-algebra, one can easily show $[[a,y],x]=0$. Thus, $[a,y]\in C_L(x)$. So, $C_L(x)$ is a left ideal of $L$.

Conversely, suppose that $C_L(x)$ is a left ideal of $L$ for all $x\in L$. It is easy to see from the Leibniz identity that $L$ is a left CL-algebra.
By a similar argument $L$ is a right CL-algebra.
\end{proof}
\begin{remark}
Theorem \ref{ideal} gives us a motivation for the conditions of CL-algebras. Note that by the conditions in Definition $3.4$, a Leibniz algebra $L$ is a left (respectively, right) CL-algebra implies
$C_L(x)$ is a left (respectively, right) ideal for all $x\in L$.
\end{remark}
\begin{cor}
For all $x\in L$, centralizers $C_L(x)$ in L are ideals of $L$ if and only if $L$ is a CL-algebra.
\end{cor}
\begin{defn}
An element $a\in L$ is said to have the left CL-property if for all $y\in C_L(x)$ and $x\in L$,
\begin{align*}
&[[x,a],y]=0, \\
&[[a,x],y]=0.
\end{align*}
An element $a\in L$ is said to have the right CL-property if for all $y\in C_L(x)$ and $x\in L$, we have
\begin{align*}
&[[x,a],y]=0, \\
&[y,[a,x]]=0.
\end{align*}
\end{defn}
\begin{defn}
An element $a\in L$ is said to have the CL-property if it satisfies both left and right CL-property.
\end{defn}
\begin{rmk}
We call an element having CL-property a CL-element. Note that the additive identity $0$ is a CL-element.
\end{rmk}
\begin{lemma}\label{CL-1}
Suppose $y\in C_L(x)$ and $a\in L$ satisfies the CL-property then $[a,y], [y,a]\in C_L(x)$.
\end{lemma}
\begin{proof}
Suppose $y\in C_L(x)$ and $a\in L$ satisfies the CL-property. From the Leibniz identity,
$$[[a,y],x]=[a,[y,x]]+[[a,x],y]=0.$$

Similarly, it is easy to check that $[x,[a,y]]=[[y,a],x]=[x,[y,a]]=0$. Thus, $[a,y],[y,a]\in C_L(x)$.
\end{proof}
\begin{theorem}
The collection S of all elements of L satisfying the CL-property
forms a Leibniz subalgebra. Thus, $(S,[~,~])$ is a CL-algebra.
\end{theorem}
\begin{proof}
To show $S$ is a Leibniz subalgebra, we need only to check that elements of $S$ are closed under the bracket of $L$. Suppose $a, b\in L$ have the CL-property. The set $S$ is non-empty as $0\in S$.
Suppose $z=[a,b]$ and $y\in C_L(x)$. From the Leibniz identity we have,
\begin{align*}
[[x,z],y]-&[[x,y],z]=[x,[z,y]],\\
&[[x,z],y]=[x,[z,y]].
\end{align*}
Since $a, b$ satisfy CL-property, it follows from Lemma  \ref{CL-1},
$$[z,y]=[[a,b],y]=[a,[b,y]]+[[a,y],b] \in C_L(x).$$ 
Similarly, $[y,z]\in C_L(x)$. We also have,
\begin{align*}
[x,[z,y]]&=0,\\
[[z,x],y]&=[z,[x,y]]-[[z,y],x]=0,\\
             [y,[z,x]]&=[[y,z],x]+[[y,x],z]=0.
\end{align*}
Thus, $z=[a,b]\in S$. Similarly, one can show that $[b,a]\in S$. Therefore, $S$ is a subalgebra of $L$.
\end{proof}
\begin{prop}\label{prop invariant}
Let $L_1$ and $L_2$ be two Leibniz algebras and $f:L_1\to L_2$ be an isomorphism between $L_1$ and $L_2$ then $f(C_L(x))=C_L(f(x))$.
\end{prop}

\begin{proof}
Let $z\in f(C_L(x))$. So $z=f(y_1)$ for some $y_1\in C_L(x)$. Now,
\begin{align*}
[z, f(x)]=[f(y_1), f(x)]=f([y_1, x])= f(0)=0.
\end{align*}
Thus, $z\in C_L(f(x))$ and $f(C_L(x))\subseteq C_L(f(x))$.

Suppose $z\in C_L(f(x))$. This implies $[f(x), z]=0$. As $f$ is an isomorphism, $z=f(y_2)$ for some $y_2\in L_1$. Now,
\begin{align*}
&[f(x), z]=0,\\
&[f(x), f(y_2)]=0,\\
&f([x, y_2])=0,\\
&[x,y_2]=0,~~~\text{Since $f$ is an isomorphism}.
\end{align*}
This implies $y_2\in C_L(x)$ and $f(y_2)\in f(C_L(x))$. So, $C_L(f(x))\subseteq f(C_L(x))$. Therefore, $f(C_L(x))=C_L(f(x))$.
\end{proof}

\begin{theorem}\label{invariant}
Let $L_1$ and $L_2$ be two Leibniz algebras and $f:L_1\to L_2$ be an isomorphism between $L_1$ and $L_2$. If $a\in L_1$ has the CL-property then $f(a)$ also has the CL-property in $L_2$.
\end{theorem}
\begin{proof}
Suppose $a\in L_1$ has the CL-property.                                                                                                                                                                                                                                                                                                                                                                                                                                          For $y\in C_L(x)$ and using Proposition \ref{prop invariant}, we have,
\begin{align*}
&[[f(x),f(a)],f(y)]=[f([x,a]),f(y)]=f([[x,a],y])=f(0)=0,\\
&[[f(a),f(x)],f(y)]=[f([a,x]),f(y)]=f([[a,x],y])=f(0)=0,\\
&[f(y),[f(a),f(x)]]=[f(y),f([a,x])]=f([y,[a,x]])=f(0)=0.
\end{align*}
Thus, the set of all CL-elements are preserved under the isomorphism of Leibniz algebras.
\end{proof}
\begin{remark}
The Theorem \ref{invariant} may be used to check whether a Leibniz morphism is an isomorphism or not.
\end{remark}


\section{Nilpotency  of CL-algebras}
\label{nil} In this section, we study how nilpotent Leibniz algebras are related to CL-algebras. We attempt to answer the following questions:
\begin{question}\label{q1}
Is every finite dimensional CL-algebra a nilpotent Leibniz algebra?
\end{question}
\begin{question}\label{q2}
Is every finite dimensional nilpotent Leibniz algebra a CL-algebra?
\end{question}
To answer Question \ref{q1}, we consider the following example.

Let $\mathbb{K}$ denotes the complex field and $L$ is a $3$ dimensional Leibniz algebra generated by the basis elements $\lbrace e_1, e_2, e_3\rbrace$ with bracket $[e_3, e_3]=e_1,~[e_3,e_2]=e_2,~[e_2, e_3]=~-e_2$ and all other product are $0$. Now,
\begin{align*}
& L^2=~<e_1,e_2>,\\
& L^3= L^4=\cdots=~<e_2>.
\end{align*}
Thus, $L$ is not nilpotent but $L$ is a CL-algebra as $C_L(e_1)=~<e_1, e_2, e_3>$, $C_L(e_2)=~<e_1, e_2>$, $C_L(e_3)=~<e_1>$.  Now,
\begin{align*}
&[[e_1,L],C_L(e_1)]=0,\,\, [[L,e_1],C_L(e_1)]=0,\,\, [C_L(e_1),[L,e_1]]=0 ;\\
& [[e_2,L],C_L(e_2)]=0,\,\, [[L,e_2],C_L(e_2)]=0,\,\, [C_L(e_2),[L,e_2]]=0 ;\\
& [[e_3,L],C_L(e_3)]=0,\,\, [[L,e_3],C_L(e_3)]=0,\,\, [C_L(e_3),[L,e_3]]=0.
\end{align*}
Thus not all CL-algebras are nilpotent.

To answer Question \ref{q2}, we consider the following theorem.
\begin{thm}\label{main theorem nil}
Nilpotent complex Leibniz algebras up to dimension $4$ are CL-algebras.
\end{thm}
\begin{proof}
Suppose $L$ is a finite dimensional complex Leibniz algebra. We consider the following four cases:\\

\underline{Case:1} 

Suppose the dimension of $L$ is $1$. So, $L$ is abelian and any abelian Leibniz algebra is automatically a CL-algebra.\\

\underline{Case:2}

Suppose the dimension of $L$ is $2$. By Theorem \ref{dim2}, up to isomorphisms there are only two nilpotent Leibniz algebras of dimension $2$. 

\resizebox{1.0 \linewidth}{!}{
  \begin{minipage}{\linewidth}
  
\begin{align*}
\mu_1:~&C_L(e_1)=<e_2>,\,\, [[e_1,L],C_L(e_1)]=0,\,\,[[L,e_1],C_L(e_1)]=0,\,\, [C_L(e_1),[L,e_1]]=0;\\
           &C_L(e_2)=<e_1,e_2>=L,\,\,[[e_2,L],C_L(e_2)]=0,\,\, [[L,e_2],C_L(e_2)]=0,\,\,[C_L(e_2),[L,e_2]]=0.
\end{align*}
\end{minipage}
}

Therefore, two dimensional nilpotent Leibniz algebras are CL-algebras.\\

\underline{Case:3}

Suppose the dimension of $L$ is $3$. By Theorem \ref{dim3}, upto isomorphisms there are only six nilpotent Leibniz algebras of dimension $3$. 

\resizebox{1.0 \linewidth}{!}{
  \begin{minipage}{\linewidth}
\begin{align*}
\lambda_2:~&C_L(e_1)=<e_2,e_3>,\,\, [[e_1,L],C_L(e_1)]=0,\,\, [[L,e_1],C_L(e_1)]=0,\,\, [C_L(e_1),[L,e_1]]=0,\\
                   & C_L(e_2)=<e_1,e_2,e_3>=L,\,\,[[e_2,L],C_L(e_2)]=0,\,\, [[L,e_2],C_L(e_2)]=0,\,\, [C_L(e_2),[L,e_2]]=0,\\
                   &C_L(e_3)=<e_1,e_2,e_3>,\,\,[[e_3,L],C_L(e_3)]=0,\,\, [[L,e_3],C_L(e_3)]=0,\,\, [C_L(e_3),[L,e_3]]=0.
\end{align*}
\begin{align*}
\lambda_3:~&C_L(e_1)=<e_1,e_3>,\,\, [[e_1,L],C_L(e_1)]=0,\,\, [[L,e_1],C_L(e_1)]=0,\,\, [C_L(e_1),[L,e_1]]=0,\\
                  & C_L(e_2)=<e_2,e_3>,\,\,[[e_2,L],C_L(e_2)]=0,\,\, [[L,e_2],C_L(e_2)]=0,\,\, [C_L(e_2),[L,e_2]]=0,\\
                  &C_L(e_3)=<e_1,e_2,e_3>,\,\,[[e_3,L],C_L(e_3)]=0,\,\, [[L,e_3],C_L(e_3)]=0,\,\, [C_L(e_3),[L,e_3]]=0.
\end{align*}
\begin{align*}
\lambda_4(\alpha):~&C_L(e_1)=<e_3>,\,\, [[e_1,L],C_L(e_1)]=0,\,\, [[L,e_1],C_L(e_1)]=0,\,\, [C_L(e_1),[L,e_1]]=0,\\
                  & C_L(e_2)=<e_3>,\,\,[[e_2,L],C_L(e_2)]=0,\,\, [[L,e_2],C_L(e_2)]=0,\,\, [C_L(e_2),[L,e_2]]=0,\\
                  &C_L(e_3)=<e_1,e_2,e_3>,\,\,[[e_3,L],C_L(e_3)]=0,\,\, [[L,e_3],C_L(e_3)]=0,\,\, [C_L(e_3),[L,e_3]]=0.
\end{align*}
\begin{align*}
\lambda_5:~&C_L(e_1)=<e_1,e_3>,\,\, [[e_1,L],C_L(e_1)]=0,\,\, [[L,e_1],C_L(e_1)]=0,\,\, [C_L(e_1),[L,e_1]]=0,\\
                  & C_L(e_2)=<e_2,e_3>,\,\,[[e_2,L],C_L(e_2)]=0,\,\, [[L,e_2],C_L(e_2)]=0,\,\, [C_L(e_2),[L,e_2]]=0,\\
                  &C_L(e_3)=<e_1,e_2,e_3>,\,\,[[e_3,L],C_L(e_3)]=0,\,\, [[L,e_3],C_L(e_3)]=0,\,\, [C_L(e_3),[L,e_3]]=0.
\end{align*}
\begin{align*}
\lambda_6:~&C_L(e_1)=<e_3>,\,\, [[e_1,L],C_L(e_1)]=0,\,\, [[L,e_1],C_L(e_1)]=0,\,\, [C_L(e_1),[L,e_1]]=0,\\
                  & C_L(e_2)=<e_2,e_3>,\,\,[[e_2,L],C_L(e_2)]=0,\,\, [[L,e_2],C_L(e_2)]=0,\,\, [C_L(e_2),[L,e_2]]=0,\\
                  &C_L(e_3)=<e_1,e_2,e_3>,\,\,[[e_3,L],C_L(e_3)]=0,\,\, [[L,e_3],C_L(e_3)]=0,\,\, [C_L(e_3),[L,e_3]]=0.
\end{align*}
\end{minipage}
}

Therefore, three dimensional nilpotent Leibniz algebras are CL-algebras.\\

\underline{Case:4}

 Suppose $L$ is four dimensional and $L=<e_1,e_2,e_3,e_4>$. By the classification  Theorem \ref{dim4} of nilpotent Leibniz algebras of dimension four, we have the following seventeen  non-isomorphic classes:
 
 \resizebox{1.0 \linewidth}{!}{
  \begin{minipage}{\linewidth}
\begin{align*}
\rho_1:~&C_L(e_1)=<e_4>,\,\, [[e_1,L],C_L(e_1)]=0,\,\, [[L,e_1],C_L(e_1)]=0,\,\, [C_L(e_1),[L,e_1]]=0,\\
                   & C_L(e_2)=<e_2,e_3,e_4>,\,\,[[e_2,L],C_L(e_2)]=0,\,\, [[L,e_2],C_L(e_2)]=0,\,\, [C_L(e_2),[L,e_2]]=0,\\
                   &C_L(e_3)=<e_2,e_3,e_4>,\,\,[[e_3,L],C_L(e_3)]=0,\,\, [[L,e_3],C_L(e_3)]=0,\,\, [C_L(e_3),[L,e_3]]=0,\\
                    &C_L(e_4)=<e_1,e_2,e_3,e_4>,\,\,[[e_4,L],C_L(e_4)]=0,\,\, [[L,e_4],C_L(e_4)]=0,\,\, [C_L(e_4),[L,e_4]]=0.
\end{align*}
\begin{align*}
\rho_2:~&C_L(e_1)=<e_4>,\,\, [[e_1,L],C_L(e_1)]=0,\,\, [[L,e_1],C_L(e_1)]=0,\,\, [C_L(e_1),[L,e_1]]=0,\\
                   & C_L(e_2)=<e_2,e_3,e_4>,\,\,[[e_2,L],C_L(e_2)]=0,\,\, [[L,e_2],C_L(e_2)]=0,\,\, [C_L(e_2),[L,e_2]]=0,\\
                   &C_L(e_3)=<e_2,e_3,e_4>,\,\,[[e_3,L],C_L(e_3)]=0,\,\, [[L,e_3],C_L(e_3)]=0,\,\, [C_L(e_3),[L,e_3]]=0,\\
                    &C_L(e_4)=<e_1,e_2,e_3,e_4>,\,\,[[e_4,L],C_L(e_4)]=0,\,\, [[L,e_4],C_L(e_4)]=0,\,\, [C_L(e_4),[L,e_4]]=0.
\end{align*}
\end{minipage}
}

\resizebox{1.0 \linewidth}{!}{
  \begin{minipage}{\linewidth}
\begin{align*}
\rho_3:~&C_L(e_1)=<e_4>,\,\, [[e_1,L],C_L(e_1)]=0,\,\, [[L,e_1],C_L(e_1)]=0,\,\, [C_L(e_1),[L,e_1]]=0,\\
                   & C_L(e_2)=<e_2,e_3,e_4>,\,\,[[e_2,L],C_L(e_2)]=0,\,\, [[L,e_2],C_L(e_2)]=0,\,\, [C_L(e_2),[L,e_2]]=0,\\
                   &C_L(e_3)=<e_2,e_3,e_4>,\,\,[[e_3,L],C_L(e_3)]=0,\,\, [[L,e_3],C_L(e_3)]=0,\,\, [C_L(e_3),[L,e_3]]=0,\\
                    &C_L(e_4)=<e_1,e_2,e_3,e_4>,\,\,[[e_4,L],C_L(e_4)]=0,\,\, [[L,e_4],C_L(e_4)]=0,\,\, [C_L(e_4),[L,e_4]]=0.
\end{align*}
\begin{align*}
\rho_4(\alpha):~&C_L(e_1)=<e_4>,\,\, [[e_1,L],C_L(e_1)]=0,\,\, [[L,e_1],C_L(e_1)]=0,\,\, [C_L(e_1),[L,e_1]]=0,\\
                   & C_L(e_2)=<e_3,e_4>,\,\,[[e_2,L],C_L(e_2)]=0,\,\, [[L,e_2],C_L(e_2)]=0,\,\, [C_L(e_2),[L,e_2]]=0,\\
                   &C_L(e_3)=<e_2,e_3,e_4>,\,\,[[e_3,L],C_L(e_3)]=0,\,\, [[L,e_3],C_L(e_3)]=0,\,\, [C_L(e_3),[L,e_3]]=0,\\
                    &C_L(e_4)=<e_1,e_2,e_3,e_4>,[[e_4,L],C_L(e_4)]=0, [[L,e_4],C_L(e_4)]=0, [C_L(e_4),[L,e_4]]=0.
\end{align*}
\begin{align*}
\rho_5:~&C_L(e_1)=<e_4>,\,\, [[e_1,L],C_L(e_1)]=0,\,\, [[L,e_1],C_L(e_1)]=0,\,\, [C_L(e_1),[L,e_1]]=0,\\
                   & C_L(e_2)=<e_2,e_3,e_4>,\,\,[[e_2,L],C_L(e_2)]=0,\,\, [[L,e_2],C_L(e_2)]=0,\,\, [C_L(e_2),[L,e_2]]=0,\\
                   &C_L(e_3)=<e_2,e_3,e_4>,\,\,[[e_3,L],C_L(e_3)]=0,\,\, [[L,e_3],C_L(e_3)]=0,\,\, [C_L(e_3),[L,e_3]]=0,\\
                    &C_L(e_4)=<e_1,e_2,e_3,e_4>,\,\,[[e_4,L],C_L(e_4)]=0,\,\, [[L,e_4],C_L(e_4)]=0,\,\, [C_L(e_4),[L,e_4]]=0.
\end{align*}
\begin{align*}
\rho_6:~&C_L(e_1)=<e_2,e_4>,\,\, [[e_1,L],C_L(e_1)]=0,\,\, [[L,e_1],C_L(e_1)]=0,\,\, [C_L(e_1),[L,e_1]]=0,\\
                   & C_L(e_2)=<e_1,e_3,e_4>,\,\,[[e_2,L],C_L(e_2)]=0,\,\, [[L,e_2],C_L(e_2)]=0,\,\, [C_L(e_2),[L,e_2]]=0,\\
                   &C_L(e_3)=<e_2,e_3,e_4>=L,\,\,[[e_3,L],C_L(e_3)]=0,\,\, [[L,e_3],C_L(e_3)]=0,\,\, [C_L(e_3),[L,e_3]]=0,\\
                    &C_L(e_4)=<e_1,e_2,e_3,e_4>,\,\,[[e_4,L],C_L(e_4)]=0,\,\, [[L,e_4],C_L(e_4)]=0,\,\, [C_L(e_4),[L,e_4]]=0.
\end{align*}
\begin{align*}
\rho_7:~&C_L(e_1)=<e_3,e_4>,\,\, [[e_1,L],C_L(e_1)]=0,\,\, [[L,e_1],C_L(e_1)]=0,\,\, [C_L(e_1),[L,e_1]]=0,\\
                   & C_L(e_2)=<e_3,e_4>,\,\,[[e_2,L],C_L(e_2)]=0,\,\, [[L,e_2],C_L(e_2)]=0,\,\, [C_L(e_2),[L,e_2]]=0,\\
                   &C_L(e_3)=<e_1,e_2,e_3,e_4>,\,\,[[e_3,L],C_L(e_3)]=0,\,\, [[L,e_3],C_L(e_3)]=0,\,\, [C_L(e_3),[L,e_3]]=0,\\
                    &C_L(e_4)=<e_1,e_2,e_3,e_4>,\,\,[[e_4,L],C_L(e_4)]=0,\,\, [[L,e_4],C_L(e_4)]=0,\,\, [C_L(e_4),[L,e_4]]=0.
\end{align*}
\begin{align*}
\rho_8:~&C_L(e_1)=<e_1,e_3,e_4>,\,\, [[e_1,L],C_L(e_1)]=0,\,\, [[L,e_1],C_L(e_1)]=0,\,\, [C_L(e_1),[L,e_1]]=0,\\
                   & C_L(e_2)=<e_3,e_4>,\,\,[[e_2,L],C_L(e_2)]=0,\,\, [[L,e_2],C_L(e_2)]=0,\,\, [C_L(e_2),[L,e_2]]=0,\\
                   &C_L(e_3)=<e_1,e_2,e_3,e_4>,\,\,[[e_3,L],C_L(e_3)]=0,\,\, [[L,e_3],C_L(e_3)]=0,\,\, [C_L(e_3),[L,e_3]]=0,\\
                    &C_L(e_4)=<e_1,e_2,e_3,e_4>,\,\,[[e_4,L],C_L(e_4)]=0,\,\, [[L,e_4],C_L(e_4)]=0,\,\, [C_L(e_4),[L,e_4]]=0.
\end{align*}
\begin{align*}
\rho_9(\alpha):~&C_L(e_1)=<e_3,e_4>,\,\, [[e_1,L],C_L(e_1)]=0,\,\, [[L,e_1],C_L(e_1)]=0,\,\, [C_L(e_1),[L,e_1]]=0,\\
                   & C_L(e_2)=<e_3,e_4>,\,\,[[e_2,L],C_L(e_2)]=0,\,\, [[L,e_2],C_L(e_2)]=0,\,\, [C_L(e_2),[L,e_2]]=0,\\
                   &C_L(e_3)=<e_1,e_2,e_3,e_4>,[[e_3,L],C_L(e_3)]=0, [[L,e_3],C_L(e_3)]=0,[C_L(e_3),[L,e_3]]=0,\\
                     &C_L(e_4)=<e_1,e_2,e_3,e_4>,[[e_4,L],C_L(e_4)]=0, [[L,e_4],C_L(e_4)]=0,[C_L(e_4),[L,e_4]]=0.
\end{align*}
\begin{align*}
\rho_{10}(\alpha):~&C_L(e_1)=<e_3,e_4>,\,\, [[e_1,L],C_L(e_1)]=0,\,\, [[L,e_1],C_L(e_1)]=0,\,\, [C_L(e_1),[L,e_1]]=0,\\
                   & C_L(e_2)=<e_3,e_4>,\,\,[[e_2,L],C_L(e_2)]=0,\,\, [[L,e_2],C_L(e_2)]=0,\,\, [C_L(e_2),[L,e_2]]=0,\\
                   &C_L(e_3)=<e_1,e_2,e_4>,\,\,[[e_3,L],C_L(e_3)]=0,\,\, [[L,e_3],C_L(e_3)]=0,\,\, [C_L(e_3),[L,e_3]]=0,\\
                     &C_L(e_4)=<e_1,e_2,e_3,e_4>,[[e_4,L],C_L(e_4)]=0,[[L,e_4],C_L(e_4)]=0,[C_L(e_4),[L,e_4]]=0.
\end{align*}
\begin{align*}
\rho_{11}:~&C_L(e_1)=<e_1,e_4>,\,\, [[e_1,L],C_L(e_1)]=0,\,\, [[L,e_1],C_L(e_1)]=0,\,\, [C_L(e_1),[L,e_1]]=0,\\
                   & C_L(e_2)=<e_3,e_4>,\,\,[[e_2,L],C_L(e_2)]=0,\,\, [[L,e_2],C_L(e_2)]=0,\,\, [C_L(e_2),[L,e_2]]=0,\\
                   &C_L(e_3)=<e_2,e_3,e_4>,\,\,[[e_3,L],C_L(e_3)]=0,\,\, [[L,e_3],C_L(e_3)]=0,\,\, [C_L(e_3),[L,e_3]]=0,\\
                     &C_L(e_4)=<e_1,e_2,e_3,e_4>,\,\,[[e_4,L],C_L(e_4)]=0,\,\, [[L,e_4],C_L(e_4)]=0,\,\, [C_L(e_4),[L,e_4]]=0.
\end{align*}
\begin{align*}
\rho_{12}:~&C_L(e_1)=<e_3,e_4>,\,\, [[e_1,L],C_L(e_1)]=0,\,\, [[L,e_1],C_L(e_1)]=0,\,\, [C_L(e_1),[L,e_1]]=0,\\
                   & C_L(e_2)=<e_2,e_3,e_4>,\,\,[[e_2,L],C_L(e_2)]=0,\,\, [[L,e_2],C_L(e_2)]=0,\,\, [C_L(e_2),[L,e_2]]=0,\\
                   &C_L(e_3)=<e_1,e_2,e_4>,\,\,[[e_3,L],C_L(e_3)]=0,\,\, [[L,e_3],C_L(e_3)]=0,\,\, [C_L(e_3),[L,e_3]]=0,\\
                     &C_L(e_4)=<e_1,e_2,e_3,e_4>,\,\,[[e_4,L],C_L(e_4)]=0,\,\, [[L,e_4],C_L(e_4)]=0,\,\, [C_L(e_4),[L,e_4]]=0.
\end{align*}
\end{minipage}
}

\resizebox{1.0 \linewidth}{!}{
  \begin{minipage}{\linewidth}
\begin{align*}
\rho_{13}:~&C_L(e_1)=<e_1,e_3,e_4>,\,\, [[e_1,L],C_L(e_1)]=0,\,\, [[L,e_1],C_L(e_1)]=0,\,\, [C_L(e_1),[L,e_1]]=0,\\
                   & C_L(e_2)=<e_2,e_3,e_4>,\,\,[[e_2,L],C_L(e_2)]=0,\,\, [[L,e_2],C_L(e_2)]=0,\,\, [C_L(e_2),[L,e_2]]=0,\\
                   &C_L(e_3)=<e_1,e_2,e_3,e_4>,\,\,[[e_3,L],C_L(e_3)]=0,\,\, [[L,e_3],C_L(e_3)]=0,\,\, [C_L(e_3),[L,e_3]]=0,\\
                     &C_L(e_4)=<e_1,e_2,e_3,e_4>,\,\,[[e_4,L],C_L(e_4)]=0,\,\, [[L,e_4],C_L(e_4)]=0,\,\, [C_L(e_4),[L,e_4]]=0.
\end{align*}
\begin{align*}
\rho_{14}:~&C_L(e_1)=<e_1,,e_3,e_4>,\,\, [[e_1,L],C_L(e_1)]=0,\,\, [[L,e_1],C_L(e_1)]=0,\,\, [C_L(e_1),[L,e_1]]=0,\\
                   & C_L(e_2)=<e_3,e_4>,\,\,[[e_2,L],C_L(e_2)]=0,\,\, [[L,e_2],C_L(e_2)]=0,\,\, [C_L(e_2),[L,e_2]]=0,\\
                   &C_L(e_3)=<e_1,e_2,e_3,e_4>,\,\,[[e_3,L],C_L(e_3)]=0,\,\, [[L,e_3],C_L(e_3)]=0,\,\, [C_L(e_3),[L,e_3]]=0,\\
                     &C_L(e_4)=<e_1,e_2,e_3,e_4>,\,\,[[e_4,L],C_L(e_4)]=0,\,\, [[L,e_4],C_L(e_4)]=0,\,\, [C_L(e_4),[L,e_4]]=0.
\end{align*}
\begin{align*}
\rho_{15}:~&C_L(e_1)=<e_1,e_3,e_4>,\,\, [[e_1,L],C_L(e_1)]=0,\,\, [[L,e_1],C_L(e_1)]=0,\,\, [C_L(e_1),[L,e_1]]=0,\\
                   & C_L(e_2)=<e_3,e_4>,\,\,[[e_2,L],C_L(e_2)]=0,\,\, [[L,e_2],C_L(e_2)]=0,\,\, [C_L(e_2),[L,e_2]]=0,\\
                   &C_L(e_3)=<e_1,e_2,e_3,e_4>,\,\,[[e_3,L],C_L(e_3)]=0,\,\, [[L,e_3],C_L(e_3)]=0,\,\, [C_L(e_3),[L,e_3]]=0,\\
                     &C_L(e_4)=<e_1,e_2,e_3,e_4>,\,\,[[e_4,L],C_L(e_4)]=0,\,\, [[L,e_4],C_L(e_4)]=0,\,\, [C_L(e_4),[L,e_4]]=0.
\end{align*}
\begin{align*}
\rho_{16}(\alpha):~&C_L(e_1)=<e_1,e_3,e_4>,\,\, [[e_1,L],C_L(e_1)]=0,\,\, [[L,e_1],C_L(e_1)]=0,\,\, [C_L(e_1),[L,e_1]]=0,\\
                   & C_L(e_2)=<e_3,e_4>,\,\,[[e_2,L],C_L(e_2)]=0,\,\, [[L,e_2],C_L(e_2)]=0,\,\, [C_L(e_2),[L,e_2]]=0,\\
                   &C_L(e_3)=<e_1,e_2,e_3,e_4>,\,[[e_3,L],C_L(e_3)]=0, [[L,e_3],C_L(e_3)]=0,[C_L(e_3),[L,e_3]]=0,\\
                     &C_L(e_4)=<e_1,e_2,e_3,e_4>,[[e_4,L],C_L(e_4)]=0, [[L,e_4],C_L(e_4)]=0,[C_L(e_4),[L,e_4]]=0.
\end{align*}
\end{minipage}
}

\resizebox{1.0 \linewidth}{!}{
  \begin{minipage}{\linewidth}
\begin{align*}
\rho_{17}:~&C_L(e_1)=<e_1,e_3,e_4>,\,\, [[e_1,L],C_L(e_1)]=0,\,\, [[L,e_1],C_L(e_1)]=0,\,\, [C_L(e_1),[L,e_1]]=0,\\
                   & C_L(e_2)=<e_2,e_3,e_4>,\,\,[[e_2,L],C_L(e_2)]=0,\,\, [[L,e_2],C_L(e_2)]=0,\,\, [C_L(e_2),[L,e_2]]=0,\\
                   &C_L(e_3)=<e_1,e_2,e_4>,\,\,[[e_3,L],C_L(e_3)]=0,\,\, [[L,e_3],C_L(e_3)]=0,\,\, [C_L(e_3),[L,e_3]]=0,\\
                     &C_L(e_4)=<e_1,e_2,e_3,e_4>,\,\,[[e_4,L],C_L(e_4)]=0,\,\, [[L,e_4],C_L(e_4)]=0,\,\, [C_L(e_4),[L,e_4]]=0.
\end{align*}
\end{minipage}
}\\

Therefore, four dimensional nilpotent complex Leibniz algebras are CL-algebras.
Thus, we have proved that nilpotent complex Leibniz algebras of dimension less than equal to four are CL-algebras.
\end{proof}
\section{Group actions and CL-algebras}
In \cite{MS}, authors have defined a notion of a finite group action on Leibniz algebra. In this section, we study the CL-property of Leibniz algebras under actions of finite groups.
\begin{defn}\label{definition-group-action}
Let $L$ be a Leibniz algebra and $G$ be a finite group. The group $G$ is said to act from the left if there exists a function 
$$\phi : G\times L \rightarrow L,~~ (g, x) \mapsto \phi (g, x) = gx$$ satisfying the following conditions.
\begin{enumerate}
\item For each $g\in G$ the map $x\mapsto gx,$ denoted by $\psi_g$ is linear.
\item $ex= x$ for all $x \in L$, where $e \in G$ is the group identity.
\item $g_1(g_2x) = (g_1g_2)x$ for all $g_1, g_2 \in G$ and $x \in L$.
\item $g[x, y] = [gx, gy]$ for all $g\in G$ and $x, y \in L.$
\end{enumerate}  
\end{defn}

The following is an alternative formulation of the above definition.
\begin{prop}
A finite group $G$ acts on a Leibniz algebra $L$ if and only if there exists a group homomorphism 
$$\psi : G \rightarrow \text{Iso}_{Leib} (L, L),~~g \mapsto \psi(g)=\psi_g,$$ from the group $G$ to the group of Leibniz algebra isomorphisms from $L$ to $L$, where $\psi_g (x) = gx$ is the left translation by $g.$  
\end{prop}
\begin{remark}
Let $G$ be a finite group and ${\mathbb K}[G]$ be the associated group ring. If $G$ acts on a Leibniz algebra $L$  then $L$ may be viewed as a ${\mathbb K}[G]$-module.
\end{remark}
\begin{defn}
Suppose $L$ is a Leibniz algebra equipped with an action of a finite group $G$.  A map $f: L\to L$ is said to be an equivariant map if $f(gx)=gf(x)$.
\end{defn}
\begin{lemma}\label{GA1}
Suppose a finite group $G$ acts on a Leibniz algebra $L$. Then there is a map 
\begin{align*}
C_L(x)\to C_L(gx).
\end{align*}
\end{lemma}
\begin{proof}
For $x\in L$ and $y\in C_L(x)$, we define a map $$p:C_L(x)\to C_L(gx),\,\,\,y\mapsto gy$$
Note that the map is well defined as $[gx,gy]=g[x,y]=g0=0$ and $[gy,gx]=g[y,x]=g0=0$. 
\end{proof}
Note that the map as defined above is just restriction of the group action map $\phi$ to $C_L(x)$.
\begin{theorem}\label{action and CL element}
A CL-element of $L$ is preserved under the action of $G$ on $L$.
\end{theorem}
\begin{proof}
Let $a\in L$ satisfies the CL-property, we need to show that for all $g\in G$, $ga$ also satisfies the CL-property.
As $a$ satisfies the CL-property, we have for all $x\in L$ and $y\in C_L(x)$
\begin{align*}
&[[x,a],y]=0, \\
&[[a,x],y]=0,\\
&[y,[a,x]]=0.
\end{align*}
From the Lemma \ref{GA1}, $gy\in C_L(gx)$. Thus, we have,
$$[[x,ga],y]=[g[g^{-1}x,a],y]=g[[g^{-1}x,a],g^{-1}y]=g0=0.$$
Similarly, one can show that 
\begin{align*}
&[[ga,x],y]=0,\\
&[y,[ga,x]]=0.
\end{align*}
So, $ga$ also satisfies the CL-property.
\end{proof}
\begin{cor}
An equivariant automorphism of Leibniz algebras preserves CL-elements.
\end{cor}
\begin{proof}
Suppose a Leibniz algebra $L$ is equipped with an action of a finite group $G$. Let $x\in L$ be a CL-element. By Theorem \ref{action and CL element}, $gx$ is also a CL-element for all $g\in G$. Suppose $f:L\to L$ is an equivariant automorphism. From the Theorem \ref{invariant} and  \ref{action and CL element}, $f(gx)=gf(x)$ is also a CL-element.
\end{proof}
\section*{Further questions}
We end this article with the following interesting questions:
\begin{enumerate}
\item In the Theorem \ref{main theorem nil}, we proved nilpotent complex Leibniz algebras up to dimension four are CL-algebras. What about nilpotent Leibniz algebras of dimension greater than four, are they CL-algebras? In general, are all nilpotent Leibniz algebras CL-algebras?
\item What is the relationship of CL-algebras with solvable Leibniz algebras? 
\end{enumerate}
\section*{Acnowledgements}
The second author would like to thank Prof. Ayupov of Institute of Mathematics, Uzbekistan Academy of Sciences and Prof. Karimbergen of Karakalpak State University, Uzbekistan for valuable discussions on this problem in a CIMPA research school on ``Non-associative algebra and applications" held at Tashkent, Uzbekistan. The second author is especially thankful to Dr. Abror Kh. Khudoyberdiyev for reading a draft version of the article and his useful suggestions. The second author expresses his gratitude to CIMPA, France for their financial help to attend the research school.

\end{document}